\documentclass[]{elsarticle}
\usepackage[english]{babel}
\usepackage[colorlinks,linkcolor=blue]{hyperref}
\usepackage{amsfonts}
\usepackage{amsmath}
\usepackage{amssymb}
\usepackage{color}
\usepackage{graphicx}
\usepackage{float}
\usepackage[mathscr]{eucal}
\usepackage{amsfonts}
\usepackage{float}
\usepackage[colorlinks,linkcolor=blue]{hyperref}
\usepackage[latin1]{inputenc}
\usepackage{amsmath}
\usepackage{graphicx}
\usepackage{amssymb}
\usepackage{amsthm}
\usepackage{verbatim}
\usepackage{epsfig}
\usepackage[labelfont=bf]{caption}
\usepackage{enumerate}
\usepackage{subfig}
\usepackage{epstopdf}

\newtheorem{lemma}{Lemma}
\newtheorem{proposition}{Proposition}
\newtheorem{corollary}{Corollary}
\newtheorem{fact}{Fact}
\newtheorem{definition}{Definition}
\newtheorem{remark}{Remark}

\newtheorem{assumption}{Assumption}

\def\begcen{\begin{center}}
\def\endcen{\end{center}}

\newcommand{\col}{ \mbox{col} }
\newcommand{\rank}{ \mbox{rank } }

\def\caly{{\cal Y}}
\def\calg{{\cal G}}

\def\calh{{\cal H}}

\def\calh{{\cal H}}
\def\calp{{\mathfrak p}}
\def\calq{{\mathfrak q}}

\def\hal{{1 \over 2}}

\def\tilthe{\tilde{\eta}}

\def\L2{{\cal L}_2}
\def\L2e{{\cal L}_{2e}}

\def\rea{\mathbb{R}}

\def\adj{\mbox{adj}}

\def\diag{\mbox{diag}}

\def\tilthe{\tilde{\theta}}

\def\x{{x}}

\def\begmat#1{\begin{bmatrix}#1\end{bmatrix}}
\def\begali#1{\begin{align}{#1}\end{align}}
\def\begalis#1{\begin{align*}{#1}\end{align*}}

\def\begequarr{\begin{eqnarray}}
\def\endequarr{\end{eqnarray}}
\def\begequarrs{\begin{eqnarray*}}
\def\endequarrs{\end{eqnarray*}}
\def\begarr{\begin{array}}
\def\endarr{\end{array}}
\def\begequ{\begin{equation}}
\def\endequ{\end{equation}}
\def\lab{\label}
\def\begdes{\begin{description}}
\def\enddes{\end{description}}
\def\begenu{\begin{enumerate}}
\def\begite{\begin{itemize}}
\def\endite{\end{itemize}}
\def\endenu{\end{enumerate}}

\def\lef[{\left[\begin{array}}
\def\rig]{\end{array}\right]}

\def\begcen{\begin{center}}
\def\endcen{\end{center}}
\def\begrem{\begin{remark}\rm}
\def\endrem{\end{remark}}
\def\begassum{\begin{assumption}}
\def\endassum{\end{assumption}}
\def\begassums{\begin{assumption*}}
\def\endassums{\end{assumption*}}
\def\begassu{\begin{ass}}
\def\endassu{\end{ass}}
\def\beglem{\begin{lemma}}
\def\endlem{\end{lemma}}
\def\begcor{\begin{corollary}}
\def\endcor{\end{corollary}}
\def\begfac{\begin{fact}}
\def\endfac{\end{fact}}
\def\TAC{{\it IEEE Trans. Automat. Contr.}}
\def\AUT{{\it Automatica}}

\def\SCL{{\it Systems and Control Letters}}

\def\tilthe{\tilde{\theta}}

\def\L2e{{\cal L}_{2e}}

\def\rea{\mathbb{R}}
\def\intnum{\mathbb{Z}}

\def\diag{\mbox{diag}}

\def\adj{\mbox{adj}}
\def\col{\mbox{col}}
\def\hal{{1 \over 2}}

\def\diag{\mbox{diag}}
\def\rank{\mbox{rank}\;}

\def\AJC{{\it Asian Journal of Control}}
\def\ARC{{\it Annual Reviews in Control}}

\def\IJACSP{{\it Int. J. on Adaptive Control and Signal Processing}}

\def\TAC{{\it IEEE Trans. Automatic Control}}

\def\EJC{{\it European Journal of Control}}

\def\SCL{{\it Systems \& Control Letters}}
\def\AUT{{\it Automatica}}

\def\CSM{{\it IEEE Control Systems Magazine}}

\def\CSL{{\it IEEE Control Systems Letters}}

\def\reapos{\mathbb{R}_{>0}}

\def\intnum{\mathbb{Z}}

\def\vect{\mbox{vec}}
\newcommand{\vecop}{\operatorname{vec}}

\def\begsubequ{\begin{subequations}}
	\def\endsubequ{\end{subequations}}
\usepackage{xcolor}

\graphicspath{{figs/}}

\input epsf

\begin{document}

\begin{frontmatter}
\title{A New Least Squares Parameter Estimator for Nonlinear Regression Equations with Relaxed Excitation Conditions and Forgetting Factor}
\author[ITAM,ITMO]{Romeo Ortega}\ead{romeo.ortega@itam.mx}
\author[ITAM]{Jose Guadalupe Romero}\ead{jose.romerovelazquez@itam.mx}
\author[SUP,ITMO]{Stanislav Aranovskiy}\ead{stanislav.aranovskiy@centralesupelec.fr}
\address[ITAM]{Departamento Acad\'{e}mico de Sistemas Digitales, ITAM, Progreso Tizap\'an 1, Ciudad de M\'exico, 04100, M\'{e}xico}
\address[SUP]{IETR--CentaleSup\'elec, Avenue de la Boulaie, 35576 Cesson-S\'evign\'e, France}
\address[ITMO]{Department of Control Systems and Robotics, ITMO University, Kronverkskiy av. 49, Saint Petersburg, 197101, Russia}

\begin{abstract}
In this note a new high performance least squares parameter estimator is proposed. The main features of the estimator are: (i) global exponential convergence is guaranteed for all  {\em identifiable} linear regression equations; (ii) it incorporates a {\em forgetting factor} allowing it to preserve alertness to time-varying parameters; (iii) thanks to the addition of a mixing step it relies on a set of {\em scalar} regression equations ensuring a superior transient performance; (iv) it is applicable to nonlinearly parameterized regressions verifying a {\em monotonicity} condition and to a class of systems with {\em switched time-varying} parameters; (v) it is shown that it is {\em bounded-input-bounded-state} stable with respect to additive disturbances; (vi) continuous and discrete-time versions of the estimator are given. The superior performance of the proposed estimator is illustrated with a series of examples reported in the literature. 
\end{abstract}
%

\end{frontmatter}

\section{Introduction}
\lab{sec1}
%
We have witnessed in the last few years an increasing interest in the analysis and design of new parameter estimators for linearly paramterized regression equations (LPRE) of the form $y(t)=\phi^\top(t)\theta$, with $y(t) \in \rea,\;\phi(t)\in \rea^q$ measurable signals and $\theta \in \rea^q$ a constant vector of unknown parameters.\footnote{We consider the case of {\em scalar} $y(t)$ to simplify the notation---as will be seen below all results can be directly extended for the case of vector $y(t)$.} The main motivation of this research is to relax the highly restrictive assumption of {\em persistent excitation} (PE) imposed to guarantee global exponential convergence of classical gradient, least squares (LS) or Kalman-Bucy algorithms \cite{GOOSINbook,SASBODbook,TAObook}. A second important motivation is to provide guaranteed good transient performance behavior since the one of the aforementioned schemes is  highly unpredictable and only a weak monotonicity property of the {\em norm} of the vector of estimation errors can be insured. 
\subsection{Review of recent literature on LS estimators}
\lab{subsec11}
%
It has recently been shown \cite{BARORT,PRA17} that global asymptotic convergence---but not exponential---of the error equation for standard continous-time (CT) gradient estimators is ensured under a strictly weaker condition of {\em generalized PE}---see \cite{BARORT,YIORT} for its definition and \cite{EFIBARORT} for some robustness properties of the algorithm. Unfortunately, this condition is still extremely restrictive to be of practical use. In \cite{BRUGOEBER} it is shown that the classical discrete-time (DT) LS algorithm is asymptotically convergent if and only if the regressor $\phi(t)$ satisfies a new excitation property, called {\em weak PE}, that is strictly weaker than PE. This result is of limited interest because, on one hand,  the definition is extremely technical and difficult to verify in applications. On the other hand, and more importantly, the analysis is limited to standard LS, without forgetting factor or covariance resetting that, as is well-known \cite{GOOSINbook,SASBODbook}, has a decreasing adaptation gain,  loosing its alertness to track parameter variations, which is the main motivation for recursive algorithms. In \cite{MARTOM} the underexcited scenario where the Gram matrix of the regressor has a $q_0$-dimensional kernel, with $q_0 \leq q$, is considered. It is shown that incorporating into a CT LS (or gradient) estimator the information of a basis expanding this kernel---the columns of the matrix $N \in \rea^{q_0 \times q}$ in equation (14) or (24) whose columns $v_i \in \rea^q,\;i=1,\dots q_0$ satisfy equation (4)---it is possible to guarantee consistent estimation to its complementary space. Clearly, if the regressor is PE the dimension $q_0$ of the aforementioned kernel is zero and convergence of all the parameters is guaranteed, confirming the well-known result of the PE case. This result is related to the {\em partial convergence} property of \cite[Theorem 2.7.4]{SASBODbook} where a similar fact is proven in the context of systems identification in underexcited situations. Although of theoretical interest, the result has little practical relevance because of the impossibilty to compute on-line the matrix $N$ mentioned above. In \cite{CUIGAUANN} a slight modification of the DT LS algorithm is proposed to deal with {\em parameter variations} in the LPRE, however the main convergence (to a compact set)  results still rely on PE assumptions. {In \cite{KRls}, an LS version of the well-known {\em I\&I estimator} \cite{IIbook} is proposed while in \cite{LOZCAN} a variation of LS that defines a {\em passive operator} is proposed.}

In \cite{SHAetal} an interesting generalization of the classical CT LS with forgetting factor is introduced, where the latter is allowed to be a time varying matrix.    See \cite{SHILEE} for a similar result in DT. The convergence analysis in both papers still relies on the PE assumption. In the very recent paper \cite{BIN} a general structure to design and analyze DT LS-like recursive estimators parameterized in terms of some free functions is proposed---see equation (7). A new definition of excitation, called $t^\star$-excitation, is given in \cite[Definition 1]{BIN}. Interestingly, this definition involves not only the regressor but also two of the aforementioned functions. A particular choice of these functions yields the classical LS estimator and $t^\star$-excitation is {\em equivalent} to PE. However, other choices of these free parameters may yield weaker excitation conditions, for instance the choice given in (27). However, this selection has again the problem of driving the adaptation gain to zero, loosing the estimator alertness and it is not clear if there are other choices that do not suffer from this drawback. Another novelty of \cite{BIN} is that it incorporates a very interesting analysis of robustness to additive disturbances in the LPRE, encrypted in the input-to-state-stability property. The interested reader is referred to \cite{BIN,SHAetal,SHILEE} for a review of the extensive literature on LS with forgetting factors.   
\subsection{Relaxing the  PE condition}
\lab{subsec12}
%
A major breakthrough in the design of recurrent estimators is the proof that it is possible to establish global convergence under the extremely weak assumption of {\em interval excitation}\footnote{It should be pointed out that IE is strictly weaker than the  generalized PE of \cite{YIORT}, the weak excitation PE property of \cite{BRUGOEBER} and the $t^\star$-excitation of \cite{BIN}.} (IE)  \cite{KRERIE}---called initial excitation in  \cite{PANetal} and excitation over a finite interval in \cite{TAObook}. To the best of the authors' knowledge the first estimators where such a result was established are the {\em concurrent}  and the {\em composite learning} schemes reported in \cite{CHOetal} and \cite{PANYU}, respectively; see  \cite{ORTNIKGER} for a recent survey on new estimators. These algorithms, which incorporate the monitoring of past data to build a stack of suitable regressor vectors, are closer in spirit to {\em off-line} estimators. See also  \cite{KRAKHA,ORTproieee} for two early references where a similar idea is explored. As is well-known, the main drawback of off-line estimators is their inability to track parameter variations, which is very often the main objective in applications. This situation motivates the interest to develop {\em bona-fide} on-line estimators that relax the PE condition preserving the scheme's alertness \cite{LJUSODbook}.

New on-line estimators relying on the use of the {\em dynamic regressor extension and mixing} (DREM) technique with weaker excitation requirements have been recently proposed. DREM was first proposed in \cite{ARAetaltac} for CT and in \cite{BELetalsysid} for DT systems. In Appendix A it is recalled that the main step in the derivation of DREM estimators is the construction of a new {\em extended} LPRE $Y(t)=\Phi(t) \theta$, with $Y(t) \in \rea^q$ and $\Phi(t)\in \rea^{q \times q}$ a new {\em square matrix} regressor. Two procedures to construct the extended LPRE, reported in \cite{KRE} and \cite{LIO}, respectively,  were originally considered---for the sake of completeness both constructions are reviewed in Appendix A. The final---and critical---{\em mixing} step consists of the multiplication of this extended LPRE by the {\em adjugate} of $\Phi(t)$.\footnote{It is interesting to note that this operation was independently reported in \cite{CHEZHAbook} in the context of stochastic estimator convergence {\em analysis}.} Clearly, this operation creates a new {\em scalar} LPRE of the form $\caly_i(t)=\Delta(t) \theta_i,\;i \in \bar q$, with $\caly(t):=\adj\{\Phi(t)\}Y(t)$ and $\Delta(t):=\det\{\Phi(t)\}$ a {\em scalar} regressor, which is the essential feature of the approach.\footnote{In Appendix A DREM is applied to nonlinearly parameterized regression equation (NLPRE) of the form \eqref{nlprect}, which is also considered in this paper.} DREM estimators have been successfully applied in a variety of identification and adaptive control problems, both, theoretical and practical ones, see \cite{ORTNIKGER,ORTetaltac} for an account of some of these results. 

The convergence properties of DREM-based estimators clearly depend on the scalar regressor $\Delta(t)$. Due to the scalar nature of $\Delta(t)$, it is clear that the parameter error converges if and only if  $\Delta(t)$ is not square integrable (summable for DT systems) and convergence is exponential if and only if  $\Delta(t)$ is PE, facts that were proven in \cite{ARAetaltac}. In \cite{GERORTNIK} a DREM-based algorithm using the extended regressor of  \cite{KRE} that ensures convergence in {\em finite-time} imposing the IE assumption on $\Delta(t)$ was proposed. An interesting open question was to establish the relation of the excitation of $\Delta(t)$ and the original regressor $\phi(t)$, which was studied in \cite{KORetalecc20} and \cite{YIORT} for the extended regressors of \cite{KRE} and \cite{LIO}, respectively. The  {\em equivalence}  between PE of  $\Delta(t)$ and PE of the original regressor  $\phi(t)$ was established for both extended regressors---proving that DREM-based estimators are at least as good as standard gradient or LS schemes for excited LPRE. On the other hand, in \cite{KORetalecc20} it is shown that if $\phi(t)$ is IE then $\Delta(t)$ is also IE for the extended regressor \cite{KRE}, while in \cite{YIORT} it was shown that the scheme of \cite{LIO} ensures the stronger property that $\Delta(t)$ is bounded away from zero in an open interval $[t_c,\infty)$ with $t_c>0$. Finally, in \cite[Proposition 3]{ORTetaltac} a new extended regressor which guarantees  exponential convergence under conditions that are {\em strictly weaker} than regressor PE was presented.

Three major developments in this line of research reported recently are:
\begite
\item[(i)] the proposal in  \cite{KORetal} and  \cite{WANetal} of two new, fully on-line, DREM-based estimators where exponential convergence is established imposing only the {\em IE condition} to $\phi(t)$; 
\item[(ii)] the proof in  \cite{WANetal} that IE of the regressor $\phi(t)$ is {\em equivalent} to identifiability of the LPRE. It should be recalled that identifiability of a LPRE is the existence of $q$ linearly independent regressor vectors \cite[Definition 2]{WANetal} $\phi(t_i),\,i \in \bar q$, and is a {\em necessary and sufficient} condition for the on- or off-line estimation of the parameters \cite{GOOSINbook};
\item[(iii)] the proof in  \cite{WANetal} that the proposed estimator is applicable also to separable {\em NLPRE} of the form \eqref{nlprect}, provided the mapping $\calg(\theta)$ verifies a monotonicity condition. Estimators for this kind of NLPRE were reported before in \cite{ORTetalaut21}, but the convergence condition was expressed in terms of the scalar regressor $\Delta(t)$.
\endite

The estimators of \cite{KORetal} and  \cite{WANetal} rely on the generation of new LPRE using the main idea of {\em generalized parameter estimation based observer} (GPEBO), which is a technique to design {\em state observers} for state-affine nonlinear systems, first proposed in \cite{ORTetal_pebo} and latter generalized in  \cite{ORTetal_gpebo}. GPEBO translates the problem of state-estimation into one of {\em parameter estimation} from a LPRE. The latter is generated exploiting the well-known property \cite[Property 4.4]{RUGbook} that the trajectories of an LTV system can be expressed as linear combinations of the columns of its fundamental matrix.  Besides the addition of the computationally demanding calculation of the fundamental matrix, a potential drawback of GPEBO is that it essentially reconstructs the {\em initial conditions} of some error equation, an operation which may adversely affect the robustness of the estimator,  \cite[Remark 7]{ORTetaltac} and  \cite{ORTajc}, see also \cite{WUetal}. The procedure followed in the construction of the estimator of  \cite{KORetal} is first the application of DREM and then invoke GPEBO, hence we refer to it as D+G. On the other hand, the estimator of \cite{WANetal} uses also GPEBO and DREM, but in the {\em opposite order}, so we refer to it in the sequel as G+D.
\subsection{Contributions of the paper}
\lab{subsec13}
%
In this paper we provide an alternative to the D+G and G+D estimators that also ensures global exponential convergence under the weak assumption of IE of the original regressor $\phi(t)$. The main features of this new estimator are summarized as follows.\\

\noindent {\bf F1} In contrast to the D+G and G+D estimators that implement a gradient descent search, we use the classical LS technique, hence we refer to it in the sequel as  LS+D estimator. The superior convergence properties of LS estimators, as opposed to gradient-based, are widely recognized \cite{GOOSINbook,LJUbook,RAOTOUbook}.\\

\noindent {\bf F2} We avoid the use of the GPEBO technique but instead exploit some structural properties of the LS estimator to construct the extended regressor. This fact removes the need to calculate the computationally demanding fundamental matrix. \\

\noindent {\bf F3} Similarly to the G+D scheme, the stability mechanism and, consequently, the stability analysis of the LS+D estimator is much more transparent than the one of the D+G estimator. There are two consequences of this fact, on one hand, the procedure of {\em tuning} the estimator to achieve a satisfactory transient performance, which is difficult for the D+G scheme, is straightforward for the LS+D one.\\

\noindent {\bf F4}   A time-varying {\em forgetting factor} that allows the estimator to preserve its alertness to time-varying parameters is incorporated.\\  

\noindent {\bf F5}  Besides the case of LPRE we consider  (separable and monotonic) {\em NLPRE}, with the associated estimator preserving all the properties of the case of LPRE. Also, we show that the proposed estimator is applicable to NLPRE with {\em switched time-varying} parameters. \\

\noindent {\bf F6} We show that the new estimator is {\em robust} with respect to additive disturbances, by proving that it defines a {\em bounded-input-bounded-state} (BIBS) stable system.\\

\noindent {\bf F7} The behaviour of many physical systems is described via CT models. On the other hand, DT implementations of estimators are of significant {practical} relevance. Therefore, similarly to  \cite{KORetal,ORTetaltac,WANetal}, to comply with both scenarios we consider in the paper both kinds of LPREs. Interestingly, in contrast to  \cite{KORetal}, the construction and analysis tools of both cases are essentially the same---however, for the sake of clarity, they are presented in separate sections.\\ 

The remainder of the paper is organized as follows. In Section \ref{sec2} we present the main result of the paper for CT systems, while the DT version is given in Section \ref{sec3}. For the sake of brevity we give both results for the general case of NLPRE, presenting the LPRE case as a corollary. Section \ref{sec4} is devoted to the derivation of the proposed extended NLPRE applying directly the DREM construction procedure. Section \ref{sec5} is devoted to the proof of {\em robustness} of the new estimator. Simulation results of some examples reported in the literature are given in Section \ref{sec6} to illustrate the superior {performance} of the proposed LS+D estimator.  The paper is wrapped-up with concluding remarks in  Section \ref{sec7}.\\

\noindent {\bf Notation.} $I_n$ is the $n \times n$ identity matrix and ${\bf 0}_{s \times r}$ is an $s \times r$ matrix of zeros. $\rea_{>0}$, $\rea_{\geq 0}$, $\intnum_{>0}$ and $\intnum_{\geq 0}$ denote the positive and non-negative real and integer numbers, respectively. For $q \in \intnum_{>0}$ we defined the set $\bar q:=\{1,2,\dots,q\}$. For $a \in \rea^n$, we denote $|a|^2:=a^\top a$, and for any matrix $A$ its induced norm is $\|A\|$. $\vect:\rea^{p \times p} \to \rea^{p^2}$ is an operator that piles up the columns of a matrix.  CT signals $x:\rea_{\geq 0} \to \rea^n$ are denoted $x(t)$, while for DT sequences $x:\intnum_{\geq 0} \to \rea^n$ we use $x_k:=x(kT_s)$, with $T_s \in \rea_{> 0}$ the sampling time. The action of an operator $\mathcal H$ on a CT signal $s(t)$ is denoted as $\mathcal H[s](t)$, and $\mathcal H[s](k)$ for a sequence $s_k$. 
%

\section{Main Result for Continuous-time Systems}
\lab{sec2}
In this section we present the proposed LS+D interlaced estimator for CT systems, with the first estimator being the  LS with bounded-gain forgetting factor proposed in \cite[Subsection 8.7.6]{SLOLIbook}. First, we consider the case of NLPRE and then specialize to LPRE that, as expected, ensures stronger convergence properties.

\subsection{Nonlinearly parameterized regression equations}
\lab{subsec21}
Consider the following CT NLPRE 
	\begequ
	\lab{nlprect}
	y(t)= \phi^\top(t)   \calg(\theta)
	\endequ
where  $y(t)  \in \rea$,  $\phi(t) \in \rea^{p}$ and $\calg:\rea^q \to \rea^p,\;q \leq p$, a smooth mapping verifying the following.\\

\noindent {\bf Assumption A1.}  [Monotonicity] There exists a matrix  $Q\in\mathbb{R}^{q\times p}$ such that mapping $\calg(\theta)$ verifies the linear matrix inequality
\begali{
	\lab{ass1}
	Q \nabla\calg(\theta) + \nabla^\top  \calg(\theta)Q^\top  \geq \rho I_q >  0,\;\forall \; \theta \in \rea^q,
}
for some $\rho \in \rea_{>0}$. Consequently  \cite{DEM,PAVetal}, The mapping $Q\calg(\theta)$ is {\em strongly monotone}, that is, 
\begali{
\lab{monpro}
(a-b)^\top  &\left[Q \calg(a) -Q \calg(b)\right] \geq \rho|a-b|^2 >  0,\;\forall \; a,b \in \rea^q,
}
with $a \neq b$.

\noindent {\bf Assumption A2.} [Interval Excitation]	The regressor   $\phi(t)$  is interval exciting (IE) \cite[Definition 3.1]{TAObook}. That is, there exists constants $C_c>0$ and $t_c>0$ such that\footnote{In \cite[Definition 3.1]{TAObook} there is an initial time in the integral that, for simplicity and without loss of generality, is taking here as zero.} 
	\begequ
	\lab{ass2}
	  \int_{0}^{t_c} \phi(s) \phi^\top(s)  ds \ge C_c I_p. 
	\endequ

\begin{proposition}\em
		\lab{pro1}
		Consider the NLPRE (\ref{nlprect}) with  $\calg(\theta)$ satisfying {\bf Assumption A1}  and $\phi(t)$ verifying {\bf Assumption A2}. Define the LS+D interlaced estimator with time-varying forgetting factor
\begsubequ
		\lab{intestt1}
		\begali{
			\lab{thegt1}
			\dot{\hat \eta}(t)   & =\alpha F (t)  \phi(t)   [y(t)  -\phi^\top (t)   \hat\eta(t) ],\; \hat\eta(0)=:\eta _{0} \in \rea^p\\
			\lab{dotf}
		\dot {F}(t)  & = -\alpha F (t) \phi(t) \phi^\top (t)   F(t) {+ \beta(t) F(t) },\;F(0)={1 \over f_0} I_p\\
			\lab{thet1c}
\dot{\hat \theta}(t)   & =\gamma Q \Delta(t)  [\caly(t)  -\Delta(t)  \calg(\hat\theta   ) ],\; \hat\theta (0)=:\theta _0 \in \rea^q	\\	\lab{dotzet}
		\dot z(t)&=-\beta(t)z(t),\;z(0)=1,
}
		\endsubequ
where we defined
		\begsubequ
		\lab{aydelt}
		\begali{
		\lab{bet}
		\beta(t)&:=\beta_0\Big(1-{\|F(t)\| \over M}\Big)\\
			\lab{delt1}
			\Delta(t)  & :=\det\{I_p- z(t)  f_0F(t)\}\\
			\lab{yt1}
			\caly(t)  & := \adj\{I_p- z(t)  f_0F(t)\}[\hat\eta(t)   - z(t)  f_0 F(t) \eta _{0}],
		}
		\endsubequ	
with tuning gains $\alpha>0$, $f_0>0,\; {\beta_0 > 0},\;M \geq {1 \over f_0}$ and $\gamma >0$. Define the parameter estimation error $\tilde \theta(t):=\hat \theta(t)- \theta$. Then, for all $f_0>0$,  $\eta _{0} \in \rea^p$ and $\theta _{0} \in \rea^q$, we have that
		\begequ
		\lab{parcon}
		\lim_{t \to \infty}\tilthe(t)  =0,\;(exp),
		\endequ
	with all signals {\it bounded}.
		\end{proposition}
 \begin{proof}
With some abuse of notation, define the signal
$$
\tilde \calg(t):=  \hat\eta(t)   - \calg(\theta),
$$
whose derivative is given by
\begali{
\lab{dottilcal}
\dot{\tilde \calg}(t)  &=-\alpha F(t)\phi (t) \phi^\top(t) \tilde\calg(t),
}
where we replaced \eqref{nlprect} in \eqref{thegt1}. Now, from
$$
{d \over dt}[F^{-1} (t)] = -  F^{-1}(t)\dot F(t) F^{-1}(t),
$$
we have that
\begequ
\lab{dotfinv}
{d \over dt}[F^{-1} (t)] =- \beta(t) F^{-1}(t)+\alpha  \phi(t) \phi^\top(t) , 
\endequ
and consequently  
\begalis{
{d \over dt}[F^{-1}(t) \tilde \calg(t) ]&  = F^{-1}(t) \dot{\tilde \calg}(t) + {d \over dt}[F^{-1}(t) ]\tilde \calg(t) \\
&= -\alpha  \phi(t)   \phi^\top(t)  \tilde \calg(t)  +  \alpha  \phi(t) \phi^\top (t)\tilde \calg(t)  \\
&- \beta(t) F^{-1} \tilde \calg(t)  \\
&= - \beta(t) F^{-1}(t) \tilde \calg(t). 
}
This implies that
\begequ
\lab{prols}
F^{-1}(t) \tilde \calg(t) =  z(t)f_0  \tilde\calg(0),\;\forall t \geq 0,
\endequ
where we used the definition of the function $z(t)$ and the fact that $F^{-1}(0)=f_0 I_p$. The latter may be rewritten as\footnote{This key identity is given in \cite[Equation (8.108)]{SLOLIbook} for the case of LPRE.}  
$$
\hat \eta(t) - \calg(\theta) = z(t)f_0F(t)  [\eta_{0} - \calg(\theta)],
$$
from which we define the {\em extended NLPRE}
\begsubequ
		\lab{keyide1}
		\begali{
		\lab{extnlpre}
		Y(t) &=\Phi(t)\calg(\theta)\\
		\lab{ylsd}
		Y(t) &:=\hat \eta(t) - z(t)f_0F (t) \eta_{0}\\
\lab{philsd}
\Phi(t) &:=I_p-z(t)f_0F(t).
}
\endsubequ
Following the DREM procedure we multiply \eqref{extnlpre} by \,  $\adj\{\Phi(t)\}$ 
to get the following NLPRE with scalar regressor
\begequ
\lab{ydelc}
\caly(t) = \Delta(t) \calg(\theta),
\endequ
 where we used \eqref{delt1} and \eqref{yt1}. Replacing \eqref{ydelc} in \eqref{thet1c} we get
\begalis{
\dot{\hat \theta}(t)  & =-{ \gamma \Delta^2(t) }Q[ \calg(\hat\theta(t) ) - \calg(\theta)] .
}
To analyse the stability of the latter system define the Lyapunov function candidate 
\begequ
\lab{lyafunv}
V(\tilde \theta(t)) := \frac{1}{2\gamma} |\tilde \theta(t)|^2.
\endequ
Computing its time derivative yields
\begalis{
\dot V(t)   & =  - \Delta^2(t) [ \hat \theta(t)- \theta]^\top Q {[ \calg(\hat \theta(t) ) - \calg(\theta)]} \\
& \leq  -  \Delta^2 (t) \rho | \tilde \theta(t)|^2  \\
& =  - 2\rho \gamma{\Delta}^2(t)  V(t) ,
}
where we invoked the {\bf Assumption A2} of strong monotonicity \eqref{monpro} of $Q \calg(\theta)$ to get the first bound.

To complete the proof we first notice that in \cite{SLOLIbook} it is shown that $F(t) \leq M I_p$ and that $\beta(t) \geq 0$---the latter implies that $z(t)$ is non-increasing and upper bounded by one. We consider two cases: when $z(t)F(t) \to 0$ and when $z(t)F(t) \geq \rho I_q>0$. In the first case we notice that  and if $z(t)F(t) \to 0$ then $\Delta(t) \to 1$ and, consequently, $\Delta(t)$ is PE. For the second case,  we solve \eqref{dotfinv} to get
$$
F^{-1}(t)= z(t)f_0 I_p +\alpha\int_0^t e^{-\int_\tau^t \beta(s)ds} \phi(\tau)\phi^\top(\tau)d\tau,\;\forall t \geq 0,
$$
which may be rewritten as
\begalis{
I_p-z(t)f_0F(t) & = \alpha F(t)\int_0^t e^{-\int_\tau^t \beta(s)ds} \phi(\tau)\phi^\top(\tau)d\tau\\
&=  \alpha F(t) z(t) \int_0^t {1 \over z(\tau)} \phi(\tau)\phi^\top(\tau)d\tau\\
& \geq   \alpha F(t) z(t) \int_0^t \phi(\tau)\phi^\top(\tau)d\tau\\
& \geq   \alpha \rho \int_0^t \phi(\tau)\phi^\top(\tau)d\tau,\;\forall t \geq 0.
}
Now, from the implication 
$$
\phi (t)   \in IE\;\Rightarrow\;  \int_0^{t}\phi(\tau)\phi^\top(\tau)d\tau >0,\;\forall t \geq t_c,
$$
we conclude that the matrix $I_p-z(t)f_0F(t)$ is nonsingular for all $ t \geq t_c$, which implies that $\Delta(t)  $ is PE. This concludes the proof.
\end{proof}

\begrem
\lab{rem1}
Notice that, as indicated above, if $z(t) \to 0$ then $\Delta(t) \to 1$ and,   in view of the extended NLPRE \eqref{extnlpre}, we have that $\caly(t)  \to \calg(\theta)$.  Consequently the parameter update law  \eqref{thet1c} will verify
$$
\dot{\hat \theta}(t)  \to \gamma Q    [\calg(\theta) -   \calg(\hat\theta(t) ) ],
$$
and the Lyapunov stability analysis still holds.
\endrem
\subsection{Linearly parameterized regression equations}
\lab{subsec22}
In the next corollary we specialize the result of Proposition \ref{pro1} for the case of LPRE. To streamline the presentation of the result we recall the following.

\begin{definition}\em \cite{GOOSINbook}
\lab{def1}
The LPRE 
\begequ
	\lab{lrect}
	y(t) = \phi^\top(t)  \theta
	\endequ
where  $y(t)  \in \rea$,  $\phi(t) \in \rea^{q}$ and $\theta \in \rea^q$ is said to be {\em identifiable} if and only if there exists a set of time instants---$\{t_i\}_{i\in \bar q},\;t_i \in \reapos$, such that
$$
\rank\Big\{\begmat{\phi(t_1)|\phi(t_2)|&\cdots&|\phi(t_q)}\Big\}=q.
$$
\end{definition}

For the sake of completeness we also recall the following result of \cite{WANetal}

\begin{lemma}\em
\lab{lem1}
The LPRE \eqref{lrect} is identifiable {\em if and only if} the regressor vector $\phi$ is IE.
\end{lemma} 

We are in position to present the main result of the subsection whose proof follows immediately from Lemma \ref{lem1}, the proof of Proposition \ref{pro1} and \cite[Proposition 2]{ORTetaltac}.

\begin{corollary}\em
\lab{cor1}
		Consider the LPRE \eqref{lrect} and assume it is {\em identifiable}. Define the LS+D interlaced estimator  with time-varying forgetting factor
		\begalis{
			\dot{\hat \eta}(t)   & =\alpha F (t)  \phi(t)   [y(t)  -\phi^\top (t)   \hat\eta(t) ],\; \hat\eta(0)=:\eta _{0} \in \rea^p\\
		\dot {F}(t)  & = -\alpha F (t) \phi(t) \phi^\top (t)   F(t) {+ \beta(t) F(t) },\;F(0)={1 \over f_0} I_p\\
\dot{\hat \theta}(t)   & =\gamma \Delta(t)  [\caly(t)  -\Delta(t)  \hat\theta    ],\; \hat\theta (0)=:\theta _0 \in \rea^q	\\	
		\dot z(t)&=-\beta(t)z(t),\;z(0)=1,
}
with   tuning gains $\alpha>0,\;f_0>0,\; {\beta \geq 0},\;M \geq {1 \over f_0}$ and $\gamma >0$,  and we used the definitions \eqref{aydelt}. Then, for all $f_0>0$,  $\eta _{0} \in \rea^q$ and $\theta _{0} \in \rea^q$, we have that \eqref{parcon} holds with all signals {\it bounded}. Moreover, the {\em individual} parameter errors verify the {\em monotonicity} condition
$$
|\tilthe_i(t_a)| \leq |\tilthe_i(t_b)|,\;\forall\, t_a \geq t_b \geq 0,\;i \in \bar q.
$$
\end{corollary}
%
\section{Main Result for Discrete-time Systems}
\lab{sec3}
In this section we present the proposed LS+D estimator for DT systems. Similarly to the previous section, we consider first the case of NLPRE and then specialize to LPRE. As will be shown below, in the DT case we need an additional Lipschitz assumption on the mapping $\calg(\theta)$, which is conspicuous by its absence in CT.

\subsection{Nonlinearly parameterized regression equations}
\lab{subsec31}
Consider the following DT NLPRE 
	\begequ
	\lab{nlprek}
	y_k= \phi_k^\top   \calg(\theta)
	\endequ
where  $y_k  \in \rea$,  $\phi_k \in \rea^{p}$ and $\calg:\rea^q \to \rea^p,\;q \leq p$. Assume $\calg(\theta)$  verifies {\bf Assumption A1} and the following.\\

\noindent {\bf Assumption A3.} [Lipschitz] The mapping $\calg(\theta)$ satisfies the {\em Lipschitz condition}
\begequ
\lab{lipcon}
|\calg(a) - \calg(b)| \leq \nu |a-b|,\; \forall a,b \in \rea^q,
\endequ
for some $\nu>0$.

Moreover, assume the regressor $\phi_k$ satisfies. 

\noindent {\bf Assumption A4.}  [Interval Excitation] 	\cite[Definition 3.3]{TAObook}	The regressor $\phi_k$ is IE. That is, there exists constants  $C_d>0$ and $k_c> 0$ such that
		\begequ
	\lab{ass4}
	\begarr{ccl} \Sigma_{j=0}^{k_c} \phi_{j} \phi^\top_{j} \ge C_d I_p.
 \endarr 
\endequ	

\begin{proposition}\em
		\lab{pro2}
		Consider the NLPRE (\ref{nlprek}) with  $\calg(\theta)$ satisfying  {\bf Assumption A1} and  {\bf  A3} and $\phi_k$ verifying {\bf Assumption A4}. Define the normalized LS+DREM interlaced estimator
		\begsubequ
		\lab{intestk}
		\begali{
			\lab{thegk}
			{\hat \eta}_{k+1} & =\hat \eta_k+{1 \over  \beta +\phi^\top_k  F_{k-1}\phi_k} F_{k-1} \phi_k  (y_k -\phi^\top_k \hat\eta_k),\\
				\lab{phit}
		{F}_k & =
{1 \over \beta}\left(I_p -{1 \over  \beta +\phi^\top_k  F_{k-1}\phi_k} F_{k-1} \phi_k \phi^\top_k \right) F_{k-1} \\
			\lab{thek}
			{\hat \theta}_{k+1} & ={\hat \theta}_{k}+\gamma Q {\Delta_k \over 1 + \Delta^2_k}[\caly_k -\Delta_k \calg(\hat\theta_k) ],\; \hat\theta_0=:\theta_0 \in \rea^q\\
			\lab{zk}
			z_k& = \beta z_{k-1},\;z_{-1}=1,
		}
		\endsubequ
with initial conditions $\hat\eta_0=:\eta_{0} \in \rea^p$, $F_{-1}={1 \over f_0} I_p$ and the definitions
		\begsubequ
		\lab{aydelk}
		\begali{
			\lab{delk}
			\Delta_k & :=\det\{I_p- f_0 z_{k} F_{k-1}\}\\
			\lab{yk}
			\caly_k & := \adj\{I_p- f_0 z_{k} F_{k-1}\} (\hat\eta_k -  f_0 z_{k} F_{k-1} \eta_{0}),
		}
		\endsubequ
		with tuning parameters the initial condition $f_0>0$, the forgetting factor $\beta \in (0,1]$ and  the adaptation gain $\gamma >0$, which is selected such that
\begin{equation}
\lab{sig}
\sigma:=\rho- {\gamma \nu^2 \over 2} \lambda_{\max}\{Q^\top Q\}>0.
\end{equation}
Define the parameter estimation error $\tilde \theta_k:=\hat \theta_k - \theta$. Then, for all $f_0>0$, $\eta_{0} \in \rea^p$ and $\theta_{0} \in \rea^q$, we have that
		\begequ
		\lab{parconk}
		\lim_{k \to \infty}|\tilthe_k|  =0,\;(exp),
		\endequ
		with all signals \textit{bounded}.
	\end{proposition}

 \begin{proof}
 To simplify the notation we define the normalization sequence
\begequ
\lab{m}
			{m}_k := \beta +\phi^\top_k  F_{k-1}\phi_k.
\endequ
With some abuse of notation, define the error signal
\begequ
\lab{tiletak}
\tilde \eta_k  :=  \hat\eta_k   - \calg(\theta),
\endequ
whose dynamics is given by
\begali{
\lab{tileta}
\tilde \eta_{k+1}  &=\left(I_p -{1 \over m_k} F_{k-1}\phi_k \phi^\top_k\right)\tilde \eta_{k},
}
where we used  \eqref{nlprek} and \eqref{thegk}. Now, direct application of the matrix inversion lemma to \eqref{phit} shows that,
\begequ
\lab{finv}
F^{-1}_k=\beta F^{-1}_{k-1} + \phi_k \phi^\top_k.
\endequ

Combining \eqref{tileta} and \eqref{finv} we can prove the following fundamental property of LS 
$$
F^{-1}_k\tilde \eta_{k+1} =\beta F^{-1}_{k-1}\tilde \eta_{k}. 
$$
Solving this difference equation we get
\begalis{
F^{-1}_{k-1}\tilde \eta_{k}& =\beta^k F^{-1}_{-1} \tilde \eta_{0}\\
 &=z_k f_0 \tilde \eta_{0},\;\forall k \geq 0, 
}
where we replaced the solution of \eqref{zk} and the initial condition choice $F_{-1}={1 \over f_0} I_p$ to get the second identity. Using the definition \eqref{tiletak}, the equation above may be rewritten as the {\em extended LPRE}
\begali{
\lab{keyide}
(I_p- f_0 z_{k} F_{k-1})\calg(\theta) &=\hat \eta_k - f_0 z_{k} F_{k-1} \eta_{0}.
}
Following the DREM procedure we multiply \eqref{keyide} by  $\adj\{I_p$ \,\,$- f_0 z_{k} F_{k-1} \}$ to get the following NLPRE
\begequ
\lab{ydel}
\caly_k= \Delta_k\calg(\theta),
\endequ
where we used \eqref{delk} and \eqref{yk}. Replacing \eqref{ydel} in \eqref{thek} we get the dynamics of the parameter error
\begali{
\lab{errequk}
\tilde { \theta}_{k+1} & =\tilde \theta_k-{ \gamma \bar \Delta_k^2}Q[ \calg(\hat\theta_k) - \calg(\theta)],
}
where, to simplify the notation, we defined the normalized scalar regressor sequence
\begequ
\lab{bardelk}
\bar \Delta^2_k:={\Delta^2_k \over 1 + \Delta^2_k} \leq 1.
\endequ
To analyze the stability of this equation define the Lyapunov function candidate
\begequ
\lab{lyafun}
V_k = \frac{1}{2 \gamma} |\tilde \theta_k|^2,
\endequ
that satisfies
\begali{
\nonumber
V_{k+1} = & V_k- \bar \Delta_k^2 \tilde\theta_k^\top Q [ \calg(\hat\theta_k) - \calg(\theta)] \\
&+\frac{\gamma}{2} \bar \Delta_k^4  [ \calg(\hat\theta_k) - \calg(\theta)]^\top Q^\top Q [ \calg(\hat\theta_k) - \calg(\theta)] \\ \nonumber
\leq & V_k - \rho \bar  \Delta^2_k |\tilde \theta_k|^2 + {\gamma \nu^2\over 2}\lambda_{\max}\{Q^\top Q\} \bar  \Delta^4_k |\tilde \theta_k|^2 \\
\leq & V_k -   \left[\rho -{\gamma \nu^2\over 2} \lambda_{\max}\{Q^\top Q\}\right] \bar \Delta^2_k|\tilde \theta_k|^2 \\ 
\nonumber
= & V_k-{\sigma} \bar \Delta^2_k   |\tilde \theta_k|^2.
\lab{vk}
}
where we invoked {\bf Assumption A2} and {\bf Assumption A3} to get the first bound, \eqref{bardelk} for the second one and used \eqref{sig} in the last identity.  Summing the inequality above we get
\begalis{
V_k-V_0 &\leq - \sum_{j=1}^k {\sigma}\bar \Delta^2_j |\tilde \theta(t)_j|^2\;\Rightarrow
{V_0 \over \sigma} \geq  \sum_{j=1}^k \bar \Delta^2_j |\tilde \theta(t)_j|^2.
}
Taking the limit as $k \to \infty$ we conclude that $\Delta_k|\tilde \theta_k| \in \ell_2$, consequently 
\begequ
\lab{limpro}
\lim_{k \to \infty}\bar  \Delta_k |\tilde \theta_k| = 0.
\endequ
 Now, from the Algebraic Limit Theorem \cite[Theorem 3.3]{RUD} we know that the limit of the product of two convergent sequences is the product of their limits. On the other hand, from the fact that
$$
V(k+1) \leq V(k) \leq V(0),\;\forall k \in \intnum_{>0},
$$ 
we have that $|\tilde \theta(k)|$ is a bounded monotonic sequence, hence it converges \cite[Theorem 3.14]{RUD}. Finally, if $\bar \Delta(k)$ converges to a {\em non-zero} limit,  we conclude from \eqref{limpro} that  $|\tilde \theta_k| \to 0$.  

We will proceed now to prove that \eqref{ass4} of  {\bf Assumption A4} ensures this property of $\Delta_k$, which together with the fact that if $\Delta_k$ converges to a non-zero limit, then $\bar \Delta_k$ also converges to a non-zero limit. Indeed, the solution of the difference equation \eqref{finv} is given by
$$
F^{-1}_k=\beta^{k+1}f_0+\beta^k \sum_{j=0}^k  \beta^{-j} \phi_j \phi^\top_j. 
$$
Evaluating this expression for $k=k_c$ yields
$$
I_p - f_0 z_{k_c+1} F_{k_c} =\beta^{k_c}F_{k_c} \sum_{j=0}^{k_c}  \beta^{-j} \phi_j \phi^\top_j. 
$$
The IE assumption ensures that the summation term is positive definite, since $F_{k_c}$ is nonsingular this ensures that the matrix on the left hand side is nonsingular.   The proof that this property holds for any $k>k_c$ stems from the observation that, for any $k_b>0$ we have that  
$$
 \sum_{j=0}^{k_c+k_b}  \beta^{-j} \phi_j \phi^\top_j= \sum_{j=0}^{k_c}  \beta^{-j} \phi_j \phi^\top_j+ \sum_{j=k_c+1}^{k_c+k_b}  \beta^{-j} \phi_j \phi^\top_j,
$$
preserving the positivity property mentioned above. This completes the proof.
\end{proof}

\subsection{Linearly parameterized regression equations}
\lab{subsec32}
In this section,  we use the result of Proposition \ref{pro2} for the case of LPRE---obviously, in this linear case {\bf Assumption A1} and {\bf Assumption A3} are automatically satisfied. As a first step we recall  \cite{WANetal} that Definition \ref{def1} and Lemma \ref{lem1}, given for continuous functions, are also valid for sequences. 

As a second step notice that for the LPRE case \eqref{keyide} takes the form
$$
(I_p- f_0 z_{k} F_{k-1})\theta =\hat \eta_k - f_0 z_{k} F_{k-1} \eta_{0},
$$
consequently \eqref{ydel} now becomes
$$
\caly_i(k)= \Delta_k \theta_i,\; i \in \bar q,
$$
and  the dynamics of the parameter error \eqref{errequk} is now given by
$$
\tilde { \theta}_{k+1}  =\tilde \theta_k-{ \gamma \bar \Delta_k^2}\tilde\theta_k ,
$$
whose stability is follows immediately from the IE assumption   \cite[Proposition 2]{WANetal} . 

\begin{corollary}\em
\lab{cor2}
		Consider the LPRE and assume it is {\em identifiable}. Define the normalized LS+D interlaced estimator with forgetting factor
		\begalis{
			{\hat \eta}_{k+1} & =\hat \eta_k+{1 \over  \beta +\phi^\top_k  F_{k-1}\phi_k} F_{k-1} \phi_k  (y_k -\phi^\top_k \hat\eta_k), \nonumber \\
		{F}_k & =
{1 \over \beta}\left(I_p -{1 \over  \beta +\phi^\top_k  F_{k-1}\phi_k} F_{k-1} \phi_k \phi^\top_k \right) F_{k-1} \nonumber \\
			{\hat \theta}_{k+1} & ={\hat \theta}_{k}+\gamma {\Delta_k \over 1 + \Delta^2_k}[\caly_k -\Delta_k  \hat\theta_k ],\; \hat\theta_0=:\theta_0 \in \rea^q, \nonumber
		}
with initial conditions $\hat\eta_0=:\eta_{0} \in \rea^q$, $F_{-1}={1 \over f_0} I_p$, tuning gains $f_0>0,\; \beta \in (0,1]$ and $\gamma >0$,  and we used the definitions \eqref{aydelk}. Then, for all $f_0>0$,  $\eta _{0} \in \rea^q$ and $\theta _{0} \in \rea^q$, we have that \eqref{limpro} holds with all signals {\it bounded}. Moreover, the {\em individual} parameter errors verify the monotonicity condition
$$
|\tilthe_i(k_a)| \leq |\tilthe_i(k_b)|,\;\forall \,k_a \geq k_b \geq 0,\;i \in \bar q.
$$
\end{corollary}

\begrem
\lab{rem2}
The importance of the {\em element-by-element} monotonicity property of the parameter error can hardly be overestimated. It played a key role for the relaxation of the assumption of known sign of the high frequency in model reference adaptive control \cite{GERORTNIK,WANetal} as well as in the solution of the adaptive pole placement problem \cite{PYRetal}. 
\endrem
%
\subsection{Switching parameters case}\label{sec:tv}
In this section, we consider the case of switched parameters estimation. Whereas the results are presented in DT only, similar results can be formulated for the CT case in a straightforward manner. Rewrite \eqref{nlprek} as
\begin{equation} \label{eq:lre:tv}
	y_k= \phi_k^\top \calg(\theta^*_{\sigma_k}),	
\end{equation}
where $\theta^*_{\sigma_k}$ denotes the switched unknown parameter vector with $\theta^*_{\sigma_k}\in\left\{\theta^*_1\right.$, $\left.\theta^*_2, \ldots, \theta^*_s\right\}$, $s\in\intnum_{>0}$. The switching signal $\sigma_k:\intnum_{\geq0}\to \bar{s}$ is a {\em known}\footnote{Such a scenario arises in several practical control scenarios, when the known switching signal $\sigma_k$ characterizes known changes in operation regimes \cite{LIBbook}.} piecewise-constant function defining the behavior of $\theta^*_{\sigma_k}$, i.e., $\theta^*_{\sigma_k} = \theta^*_i$ when $\sigma_k=i$, $i\in\bar{s}$. The known time instants when $\sigma_k$ changes its value are further denoted as $t_{r,i}$, $i\in\intnum_{\geq0}$.

The estimator \eqref{intestk}, \eqref{aydelk} is not capable of estimating switched parameters as for $\beta<1$ the sequence $z_k$ converges to zero, and for $\beta=1$ the LS estimator looses its alertness. To deal with switching parameters, we propose a resetting-based modification of the estimator \eqref{intestk}, \eqref{aydelk}:

\begin{subequations} \label{est:resetting}
	\begin{align}
		{\hat \eta}_{k+1} & =\hat \eta_k+\frac{1}{1 +\phi^\top_k F_{k-1}\phi_k} F_{k-1} \phi_k  (y_k -\phi^\top_k \hat\eta_k),
		\\
		{\hat \theta}_{k+1} & ={\hat \theta}_{k}+\gamma Q \frac{\Delta_k}{1 + \Delta^2_k}[\caly_k -\Delta_k \calg(\hat\theta_k) ],\; \hat\theta_0=:\theta_0 \in \rea^q,
		\\
		{F}_k & = 
		\begin{cases}
			{\mathcal N}_k F_{k-1}& \mbox{ if } t_k \ne t_{r,i}\; \forall i, \\
			\frac{1}{f_0}I_p & \mbox{ otherwise,} 
		\end{cases}
		\\
		{\mathcal N}_k&=\left(I_p -\frac{1}{1 +\phi^\top_k  F_{k-1}\phi_k} F_{k-1} \phi_k \phi^\top_k \right), \nonumber \\
		\psi_{k+1} & = 
		\begin{cases}
			\psi_{k} & \mbox{ if } t_k \ne t_{r,i}\; \forall i, \\
			{\hat \eta}_{k+1} & \mbox{ otherwise,} 
		\end{cases}, \quad \psi_0 = \eta_0,
	\end{align}
\end{subequations}
where  $\hat\eta_0=:\eta_{0} \in \rea^p$, $F_{-1}={1 \over f_0} I_p$, and 
\begin{subequations} \label{est_def:resetting}
	\begin{align}
		\Delta_k & =\det\{I_p- f_0 F_{k-1}\},
		\\
		\caly_k & := \adj\{I_p- f_0 F_{k-1}\} (\hat\eta_k -  f_0 F_{k-1} \psi_{k}). \label{calY:reset}
	\end{align}
\end{subequations}

Between the reseting instances $t_{r,i}$, the estimator \eqref{est:resetting}, \eqref{est_def:resetting} reproduces the estimator \eqref{intestk}, \eqref{aydelk} with $\beta=1$ and thus $z_k\equiv 1$. Then, at each reset instance $t_{r,i}$, the matrix $F_k$ is reset to its initial condition $F_{-1}$, and the state $\psi_k$ saves the value of $\hat{\eta}_k$. The state $\psi_k$ thus plays the same role as $\eta_0$ in \eqref{keyide}, compare \eqref{calY:reset} and \eqref{yk}. Following the properties of \eqref{intestk}, \eqref{aydelk}, the proposed estimator ensures the boundedness of the states and is capable of estimating $\theta^*_{\sigma_k}$ if the following assumption holds.

\noindent {\bf Assumption A5.} [Switching Interval Excitation].
The switching signal $\sigma_k$ is such that the regressor $\phi_k$ is IE between two subsequent switching instants. That is, there exist constants $C_d>0$ and $k_c> 0$ such that for any $i\in \intnum_{\ge0}$
\[
	t_{r,i}+k_c \le t_{r,i+1}
\]
and
\[
	\sum_{\ell=0}^{k_c} \phi_{t_{r,i}+\ell} \phi^\top_{t_{r,i}+\ell} \ge C_d I_p.
\]
\begrem
\lab{rem3}
In words, {\bf Assumption A5} means that the regressor satisfies the IE condition inside each subinterval $[t_{r,i},t_{r,i+1}]$. For simplicity we have taken that the constants $k_c$ and $C_d$ that appear in the definition of IE are the {\em same} for all subintervals $[t_{r,i},t_{r,i+1}]$, but this is clearly not necessary. 
\endrem
\section{Derivation of the Extended NLPRE \eqref{keyide1} via DRE}
\label{sec4}
%
To simplify the reading of the material presented in this section we refer the reader to Appendix A  where the procedure to derive DREM is recalled. 

In Proposition \ref{pro1} it is shown that the dynamic extension \eqref{thegt1} and \eqref{dotf} generates the extended NLPRE  \eqref{keyide1} 
to  which we apply the mixing step {\bf S4} of  Appendix A to generate the scalar NLPRE \eqref{ydelc}. In this section we prove that this extended NLPRE can also be derived directly applying the DREM step {\bf S2} of  Appendix A  for a suitably defined LTV operator $\calh$.\footnote{We refer the interested reader to \cite[Proposition 3]{WANetal} where the DREM operator $\calh$ for the G+D estimator reported in \cite[Proposition 2]{WANetal} is identified.} For the sake of brevity we only consider the CT case, with the DT case following {\em verbatim}.

\begin{proposition}\em
		\lab{pro3}
Define the state space realization of the LTV operator $\calh:u \to U$ used in step {\bf S2} of Appendix A as in \eqref{dotu} with 
$$
A(t):=-\alpha F(t) \phi(t) \phi^\top(t),\;b(t):=\alpha F(t) \phi(t),
$$
with $F(t)$ defined in \eqref{dotf}. Starting from the NLPRE $y(t) =\phi^\top(t) \calg(\theta)$, construct $Y(t) \in \rea^p$ and $\Phi(t) \in \rea^{p \times p}$ via \eqref{yphiapp} that is, as the solutions of the dynamic extension
\begsubequ
\lab{nlpren}
\begali{
\dot Y(t) &=-\alpha F(t)  \phi (t) \phi^\top(t)  Y(t)+\alpha F(t)  \phi(t)  y(t),\\
\dot \Phi(t) &=-\alpha F(t)  \phi (t) \phi^\top(t)  \Phi(t) + \alpha F(t)  \phi (t) \phi^\top(t),
}
\endsubequ
and initial conditions $Y(0)={\bf 0}_{p \times 1}$ and $\Phi(0)={\bf 0}_{p \times p}$.
\begenu
\item The extended NLPRE $Y(t) =\Phi(t)\calg(\theta)$ holds. 
\item The signals $Y(t) $ and $\Phi(t) $ satisfy \eqref{ylsd} and \eqref{philsd}, respectively, with $\hat \eta(t)\in \rea^p$ and $F(t) \in \rea^{p \times p}$ solutions of the differential equations \eqref{thegt1} and \eqref{dotf}, respectively.
\endenu
\end{proposition}

\begin{proof}
The fact that the extended NLPRE $Y(t)=\Phi(t)\calg(\theta)$ holds follows trivially from linearity of the operator $\calh$. 

To prove the claim (ii) we invoke \eqref{ylsd} and do the following calculations 
\begalis{
\dot Y(t)&=\dot{\hat \eta}(t) -z(t) \dot F(t) \eta _{0}- \dot z(t)f_0 F(t) \eta _{0}\\
&=\alpha F(t)\phi(t) [y(t)  -\phi^\top(t)    \hat\eta(t)]+ \beta(t) z(t)f_0 F(t) \eta _{0}\\
&\;\;\;+z(t)[ \alpha F (t) \phi(t) \phi^\top (t)   F(t) {- \beta(t) F(t) }] \eta _{0} \\
&=\alpha F(t)\phi(t) y(t)-\alpha F(t) \phi (t)\phi^\top(t) [ \hat\eta(t)-z(t) f_oF(t) \eta _{0}]\\
&=A(t)Y(t)+b(t)y(t).
}
In the same spirit as above we compute the time derivative of $\Phi(t)$ as defined in \eqref{philsd} to get
\begalis{
\dot \Phi(t)&=-\dot z(t)f_o  F(t) -  z(t)f_0 \dot  F(t) \\
&=  \alpha z(t) f_o F(t)\phi(t) \phi^\top (t)   F(t)\\
&=\alpha F(t)\phi(t) \phi^\top(t) -\alpha F(t) \phi(t) \phi^\top(t) [I_q-z(t) f_oF(t) ]\\
&=A(t)\Phi(t)+b(t)\phi^\top(t).
}
This completes the proof.
\end{proof}

\begrem
\lab{rem4}
It is important to note that the relation $Y(t)= \Phi(t)\calg(\theta)$ imposes the constraint $Y(0)= \Phi(0) \calg(\theta)$, which is satisfied with the zero initial conditions imposed  in Proposition  \ref{pro3}. As expected, this choice is consistent with the choice of initial conditions for $\hat \eta(t)$ and $F(t) $ given in Proposition  \ref{pro1}.
\endrem

\begrem
\lab{rem5}
The dynamic extension \eqref{dotf}and  \eqref{nlpren} provides an alternative to the construction of the proposed estimator. The relationship between the two implementations boils down to a standard diffeomorphic change of coordinates.  Indeed, while the state of the system in  \eqref{intestt1} and \eqref{dotzet} is given by $\col(\hat \eta(t),\vecop(F(t)),$ $z(t),\hat \theta(t)) \in \rea^{(p+p^2+1+q)}$, the state of the system of Proposition \ref{pro3} is  $\col(Y(t)$, $\vect(\Phi(t)), z(t),\hat \theta(t))\in \rea^{(p+p^2+1+q)}$, and the first two components are related by a simple invertible coordinate change
$$
\begmat{\hat \eta(t) \\ \vect(F(t))}=\begmat{Y(t) +[I_p-\Phi(t)] \eta_0\\ {1 \over z(t)f_0}\vect(I_p-\Phi(t))}.
$$
However, the original implementation \eqref{intestt1} clearly reveals the mechanism underlying the operation of the estimator, namely, the use of a classical LS update and the creation of the extended NLPRE exploiting the well-known property of LS \eqref{prols}.\footnote{To the best of the authors' knowledge, this property  was first reported in \cite[equation (17)]{DEL} and was widely used for the implementation of projections in indirect adaptive controllers \cite{LOZZHA}.}
\endrem
\section{Robustness Analysis of the CT LS+D Estimator}
\lab{sec5}
%
In this section we analyze the robustness {\em vis-\`a-vis} additive perturbations of the CT LS+D estimator of Proposition \ref{pro1}. That is, we consider the {\em perturbed} NLPRE
\begequ
\lab{pernlpre}
  y(t)=   \phi^\top(t)  \calg(\theta) + d(t),
\endequ
where $d(t)$ represents an additive perturbation signal. This signal may come from additive noise in the measurements of $y(t)$ and $\phi(t)$ or time variations of the parameters, that is, $d(t)$ may be decomposed as
$$
d(t)=d_y(t) + d^\top_\theta(t) \phi(t) + d^\top_\phi(t) \calg(\theta),
$$
where $d_y(t)\in \rea$ and  $d_\phi(t) \in \rea^p$ represent the measurement noise added to $y(t)$ and $\phi(t)$, respectively, and $d_\theta(t) \in \rea^p$ captures {\em time variations} in the parameters. We make the reasonable assumption that these signals are all bounded and prove that the CT LS+D estimator defines a {\em bounded-input-bounded-state} (BIBS) stable system.

The main result is summarized in the proposition below.

\begin{proposition}
\lab{pro4}\em
Consider  the {\em perturbed} NLPLPRE \eqref{pernlpre} with $d(t)$ a bounded signal. Assume the regressor $\phi(t)$ is IE. The LS+D estimator of Proposition \ref{pro1} applied to this NLPRE is {\em BIBS stable}.
\end{proposition}

\begin{proof}
In the light of Remark \ref{rem3}, to carry out the proof we rely on the use of the alternative implementation of the extended NLPRE of Proposition \ref{pro3}. Applying the operator $\calh$ of Propositions \ref{pro3}  to the perturbed NLPRE \eqref{pernlpre} yields the perturbed version of the extended LPRE \eqref{extnlpre} as
\begin{equation}\label{eq:per-LPRE}
Y(t)=\Phi (t)\calg(\theta) + \calh[d](t),
\end{equation}
where we exploited the property of linearity of $\calh$. Next we proceed to show that the operator $\calh$ is BIBO-stable. This is done by proving that, for all bounded $d(t)$, the signal $\calh[d](t)$ is also bounded.

The signal $\calh[d](t)$ is generated via the CT LTV system
\begalis{
\dot x_d(t)& = -\alpha  F  \phi(t) \phi^\top (t) x_d(t) + \alpha F(t)\phi(t)d(t)\\
\calh[d](t)&=x_d(t).
}
Defining $W(x_d) :=\hal  x_d^\top(t) F^{-1}(t) x_d(t)$, we have
\begalis{
\dot W & =-\beta(t) W(t)-{\alpha \over 2}[\phi^\top(t)x_d(t) - d(t) ]^2+ {\alpha \over 2} d^2(t)\\
& \leq -\beta(t) W(t)+ {\alpha \over 2} d^2(t).
}
As  $\beta(t)>0$ and $d(t)$ is bounded, this proves that $x_d(t)=\calh[d](t)$, is also bounded. 

From the analysis above, we conclude that the operator $\calh$ is BIBO-stable. Consequently, since $\phi(t)$ and $y(t)$ are bounded, it follows that $Y(t)=\calh[y](t)$ and   $\Phi(t)=\calh[\phi^\top](t)$ are also bounded. It only remains to prove that  and $\hat \theta(t)$ is bounded. Whence, multiplying \eqref{eq:per-LPRE} by   $\adj\{\Phi(t)\}$ we get the following perturbed NLPRE
\begequ
\lab{ydelper}
\caly(t) = \Delta(t)\calg(\theta)+\xi(t),
\endequ
where we defined the signal
\begequ
\lab{pertilthe}
\xi(t):= \adj\{\Phi(t)\}x_d(t).
\endequ
We notice that this signal is bounded. Replacing \eqref{pertilthe} in the estimator \eqref{thet1c} yields 
\begalis{
\dot{\hat \theta}(t)  & =-{ \gamma \Delta^2(t) }Q[ \calg(\hat\theta(t) ) - \calg(\theta)] +{ \gamma \Delta(t) }Q\xi(t).
}
Computing the derivative of  the Lyapunov function candidate \eqref{lyafunv} we get
\begalis{
\dot V(t)   & =  - \Delta^2(t) [ \hat \theta(t)- \theta]^\top Q {[ \calg(\hat \theta(t) ) - \calg(\theta)]} \\
&\;\;\;+  \Delta(t) \tilde \theta^\top(t) Q\xi(t)\\
& \leq   - \rho {\Delta}^2(t) | \tilde \theta(t)|^2 +  \Delta(t) |\tilde \theta(t)|| Q\xi(t)| \\
& =   - {\rho \over 2}{\Delta}^2(t) | \tilde \theta(t)|^2  - {\rho \over 2}[{\Delta}(t) | \tilde \theta(t)| -{1 \over \rho}| Q\xi(t)|]^2\\
&\;\;\;+{1 \over 2\rho}| Q\xi(t)|^2\\
& \leq    - \rho \gamma{\Delta}^2(t)  V(t) +{1 \over 2\rho}| Q\xi(t)|^2.
}
The proof of boundedness of $\tilde \theta(t)$ is completed recalling that in Proposition \ref{pro1} it is shown that $\Delta(t)$ is PE.
\end{proof}
\section{Simulation Examples}
\lab{sec6}
In this section we present simulations of the proposed CT and DT estimators  using different examples recently reported in the literature.

\subsection{Example 5 of \cite{MARTOM}}
Consider the second order stable, CT, linear system described by 
\begin{align*}
\dot x_1(t)=&x_2(t) \\
\dot x_2(t)=& -\theta_1 x_1(t) -\theta_2 +\theta_3 u(t) \\ y(t)=&x_1(t),
\end{align*}
or equivalently 
\begin{equation}
\ddot x_1(t) = -\theta_1  x_1(t)-\theta_2 \dot x_1(t)  +\theta_3 u(t)
\label{ex5MT}
\end{equation}
where $\theta_1$, $\theta_2$ and $\theta_3$ are unknown parameters. Applying the filter 
$$
H(\calp)={1 \over {\calp+\lambda}}
$$
where $\calp:=\frac{d}{dt}$,  to both sides of \eqref{ex5MT} and rearranging the terms, we get the LPRE \eqref{lrect} with 
\begin{equation}
\label{yphi}
y(t)=  \calp H({\calp}) [ x_2](t), \quad \phi(t)=H({\calp})[\col( -x_1(t), \, -\calp x_1(t), \, u(t) )],
\end{equation}
and $\theta:=\col(\theta_1, \, \theta_2, \, \theta_3)$.

To carry out the simulations we use the same conditions that \cite{MARTOM}, that is, we set to zero the initial conditions of the filters, as well as the initial value of the parameter estimation vector ${\hat \theta} (0)=0$, $\hat \eta(0)=\col(0.1,\,0.1, \, 0.1)$, $u(t)=5$ and fix $\theta=\col(2,3,1)$. Besides, the tuning parameters of the proposed estimator of Corollary \ref{cor1} were  $\alpha=20.3$, $f_0=4$, $\beta=0.07$ and $\gamma=700$. In Fig. \ref{fig2} we appreciate the transient behavior of the estimated parameters, which clearly shows the estimation of the real values. This result should be contrasted with the non-converging behavior of the estimates reported in  \cite{MARTOM} with the gradient scheme and their modified gradient.

\begin{figure}[htp]
\centering
  \includegraphics[width=.85\linewidth]{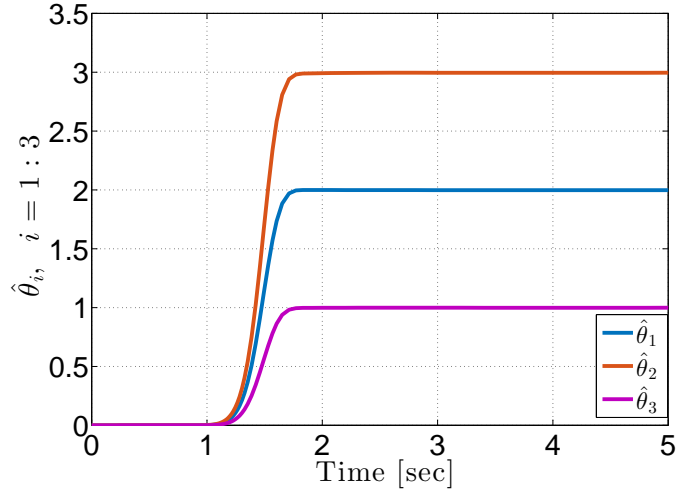}  
  \caption{Transient behavior of the estimated parameters  $\hat \theta_i(t)$ with $i=1,2,3$.}
 \label{fig2}
\end{figure}

To illustrate the use of the NLPRE \eqref{nlprect}, we notice that from the proposed values for $\theta$, we have that $\theta_3$ can be rewritten as $\theta_3=\theta_2-\theta_1$. Hence, after the application of the filter $H(\calp)$, the system \eqref{ex5MT}  can be written as the NLPRE \eqref{nlprect} with ${\mathcal G} (\theta):= \col(\theta_1,\, \theta_2, \, \theta_2-\theta_1)$, Thus, using the same initial conditions,  estimator gains and verifying {\bf Assumption A1} with 
$$
Q= \left[\begin{array}{ccc}  1  & 0 & 0 \\ 0   & 1 & 0 \end{array}\right],
$$ 
and $\rho=1$, we carry out a simulation to estimate only $\theta_1$ and $\theta_2$ with the estimator of Proposition \ref{pro1}. Fig. \ref{fig3} shows the transient behavior of the estimated parameters, showing again parameter convergence.

\begin{figure}[htp]
\centering
  \includegraphics[width=.85\linewidth]{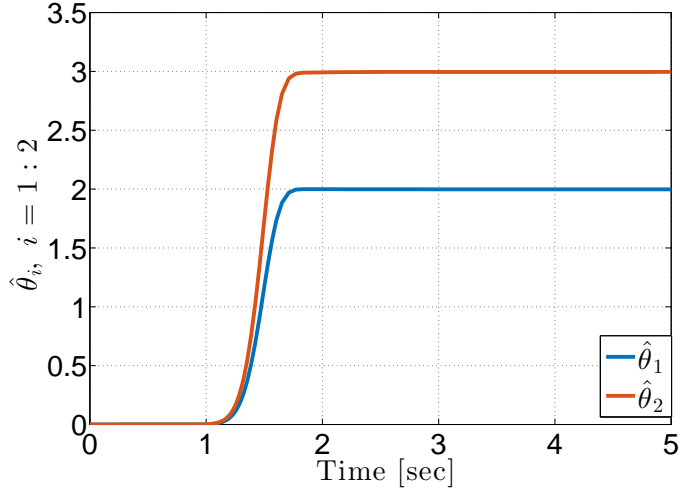}  
  \caption{Transient behavior of the estimated parameters  $\hat \theta_i(t)$ with $i=1,2$ using the NLPRE \eqref{nlprect}.}
 \label{fig3}
\end{figure}

\subsection{Example  4 of \cite{GOEBRUBER}}

Consider the first order linear system 
\begin{equation}
\label{ex4BER}
y_k= \frac{\theta_2}{ \calq - \theta_1} u_k,
\end{equation}
 where $\calq$ is the forward-shiff operator and $\theta_1$ and $\theta_2$ are unknown parameters.  After some simple calculations, we have that \eqref{ex4BER} can be written as  a LPRE $y_k=\phi_k^\top \theta$ with 
 $$
 \phi_k=\col(y_{k-1}, \, u_{k-1}), \quad \theta=\col(\theta_1, \, \theta_2).
 $$
 To carry out the simulations we have also used the same initial conditions and parameters  of \cite{GOEBRUBER}, that is, $\theta_1=0.4$, $\theta_2=0.8$, $\theta_0=\col(0,\,0)$ and the input signal $u_k=1$.\footnote{We notice that there is an unfortunate typo in the definition of $u_k$ in \cite[Example  4]{GOEBRUBER}.}  The tuning gains of the estimator of Corollary \ref{cor2} were chosen as $\beta=1$, $f_0=0.14$, $\gamma=0.4$ and initial conditions ${\hat \eta}_0=\col(1,\,1)$.  It is important to note that for this system \eqref{ex4BER}, the estimator proposed in \cite{GOEBRUBER} only ensures the boundedness of $\hat \theta_i$, with $i=1,2$ (see Fig. 3 of \cite{GOEBRUBER}).  This should be contrasted with our estimator, which, as can be seen in Fig. \ref{fig4}, converges to the real value. 

\begin{figure}[htp]
\centering
  \includegraphics[width=0.82\linewidth]{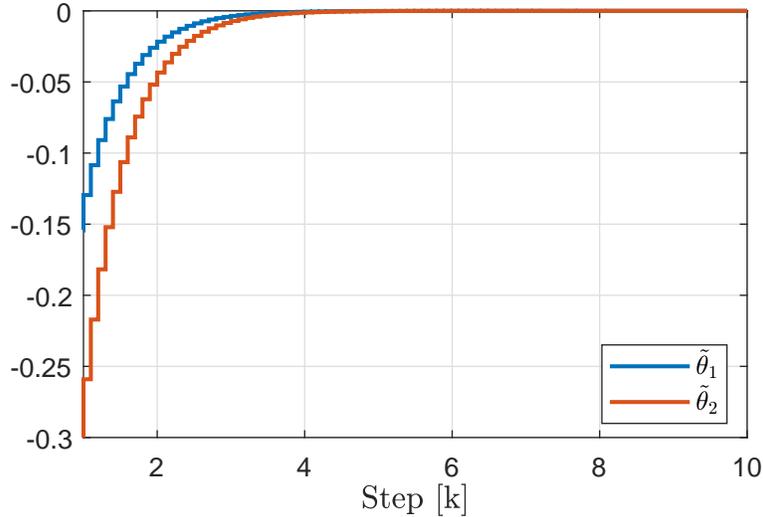}  
  \caption{Transient behavior of the parameter estimates  $\hat \theta_i$ with $i=1,2$.}
 \label{fig4}
\end{figure}

\subsection{Example 8 of \cite{NIBK}}
We consider the DT system 
$$
y_{k} = -0.5y_{k-1} + 0.1y_{k-2} + u_{k-1} - 0.4u_{k-2},
$$
which switches for $k \ge k_c$ to 
$$
y_{k} = 1.4y_{k-1} - 0.3y_{k-2} + u_{k-1} - 1.3u_{k-2}.
$$
Note that for $k\ge k_c$ the plant is unstable and not minimum-phase. The initial conditions are $y_{-1}=-0.2$, $y_{-2} = 0.4$, and $u_{-1}=u_{-2} = 0$, and $k_c=50$.

The system can be written in the form \eqref{eq:lre:tv} setting $\calg(\theta^*_{\sigma_k}) = \theta^*_{\sigma_k}$,
\[
	\begin{gathered}
		\phi_k := \begin{bmatrix} -y_{k-1} & -y_{k-2} & u_{k-1} & u_{k-2} \end{bmatrix}^\top, \\
		\theta^*_1 := \begin{bmatrix} 0.5 & -0.1 & 1 & -0.4 \end{bmatrix}^\top, \\
		\theta^*_2 := \begin{bmatrix} -1.4 & 0.3 & 1 & -1.3 \end{bmatrix}^\top,
	\end{gathered}
\]
and 
\[
	\sigma_k = \begin{cases}
		1 & \mbox{for } k < k_c, \\
		2 & \mbox{for } k \ge k_c,
	\end{cases}
\]
which corresponds to $t_{r,0}=0$ and $t_{r,1} = k_c$.

In this example, we consider the indirect adaptive poles placement for the reference tracking, where the reference signal is denoted as $r_k$. Then the control signal $u$ is given by
\[
	u_k = c_{1,k}y_k + c_{2,k} y_{k-1} + c_{3,k} u_{k-1} + c_{4,k} r_k,
\]
where the time-varying coefficients $c_{1,k}$, $c_{2,k}$, $c_{3,k}$, and $c_{4,k}$ are computed based on the current parameter estimate $\hat{\theta}_k$ to provide the desired poles and unit gain of the closed-loop system; if for a value of $\hat{\theta}_k$ the computations are ill-conditioned, then $u_k=0$ is chosen.  For this example, the desired poles for this are $e^{-1}$, $e^{-0.5+0.86\sqrt{-1}}$, and $e^{-0.5-0.86\sqrt{-1}}$, and the reference signal is $r_k\equiv 1$.

To estimate the parameters, we apply the resetting-based estimator \eqref{est:resetting}, \eqref{est_def:resetting}, where we set $f_0=0.4$, $\gamma= 500$, $\eta_0 = 0$, and $\theta_0 = \begin{bmatrix} 0.1 & -0.3 & 0.5 & -0.05 \end{bmatrix}$. Note that $\theta_0$ cannot be chosen zero as such a choice yields zero input to the system and the regressor $\phi$ is not IE; for a nonzero choice of $\theta_0$, the interval excitation is provided by the transients of the plant. 

The simulation results are depicted in Fig.~\ref{fig:transients:y} for the output signal $y$ and in Fig.~\ref{fig:transients:theta} for the estimation errors $\tilde{\theta}_k = \hat{\theta}_k - \theta^*_{\sigma_k}$. It can be observed that after the switch, parameters estimation errors remain almost constant for approximately $30$ steps, and then quickly converge. Further investigation shows that the regressor $\phi$ is not exciting on this initial interval, and thus the estimation does not progress. As soon as the IE condition is satisfied, the estimates $\hat{\theta}_k$ converges to the true value $\theta^*_2$.

%
%
%
%

\begin{figure}[htb]
	\subfloat[Initial transients]{%
		\includegraphics[width=0.48\linewidth]{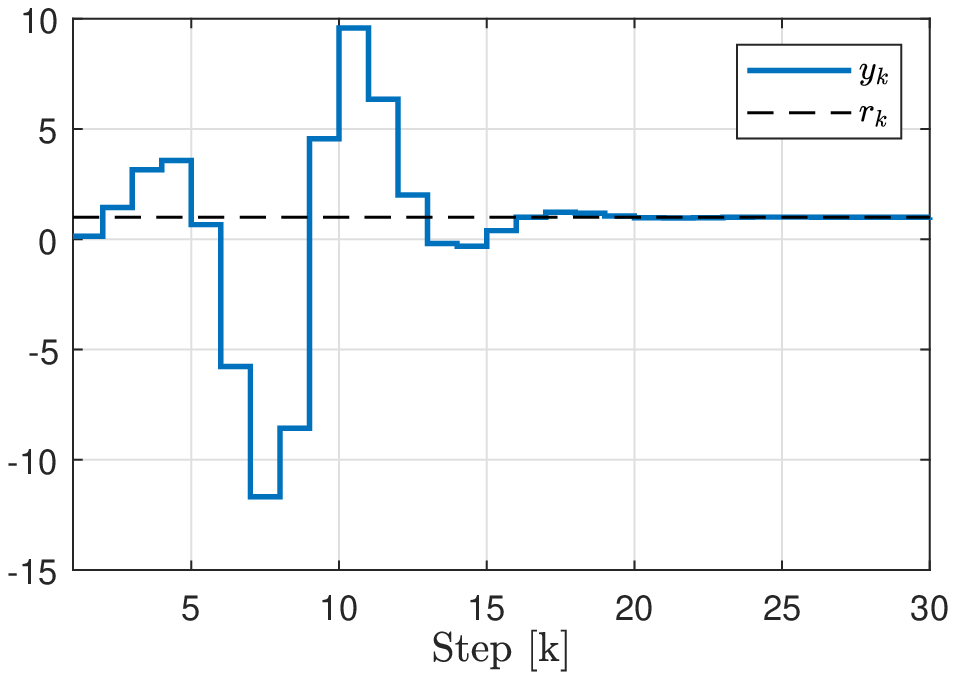}
	}
	~
	\subfloat[Transients after the switch, $k_c = 50$]{%
		\includegraphics[width=0.48\linewidth]{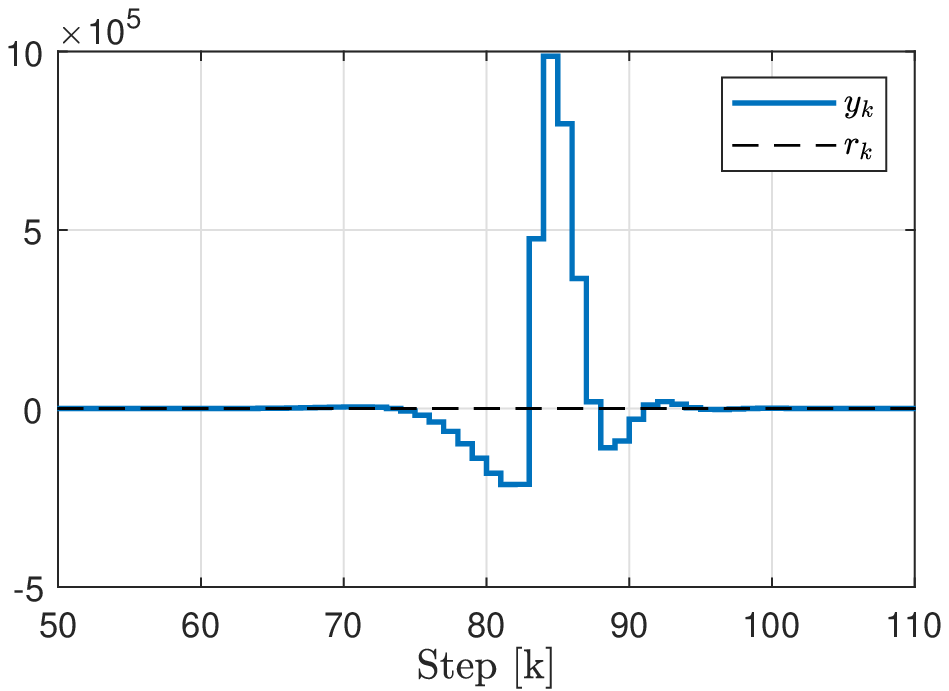}
	}
	\caption{Transients of $y_k$.}
	\label{fig:transients:y}
\end{figure}
\begin{figure}[htb]
	\subfloat[Initial transients]{%
		\includegraphics[width=0.48\linewidth]{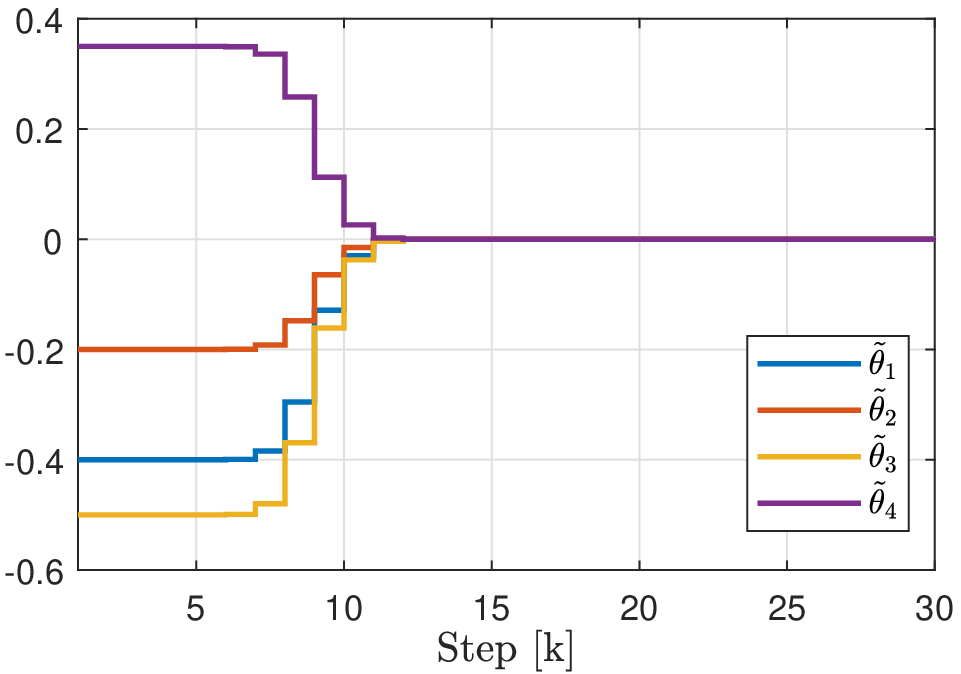}
	}
	~
	\subfloat[Transients after the switch, $k_c = 50$]{%
		\includegraphics[width=0.48\linewidth]{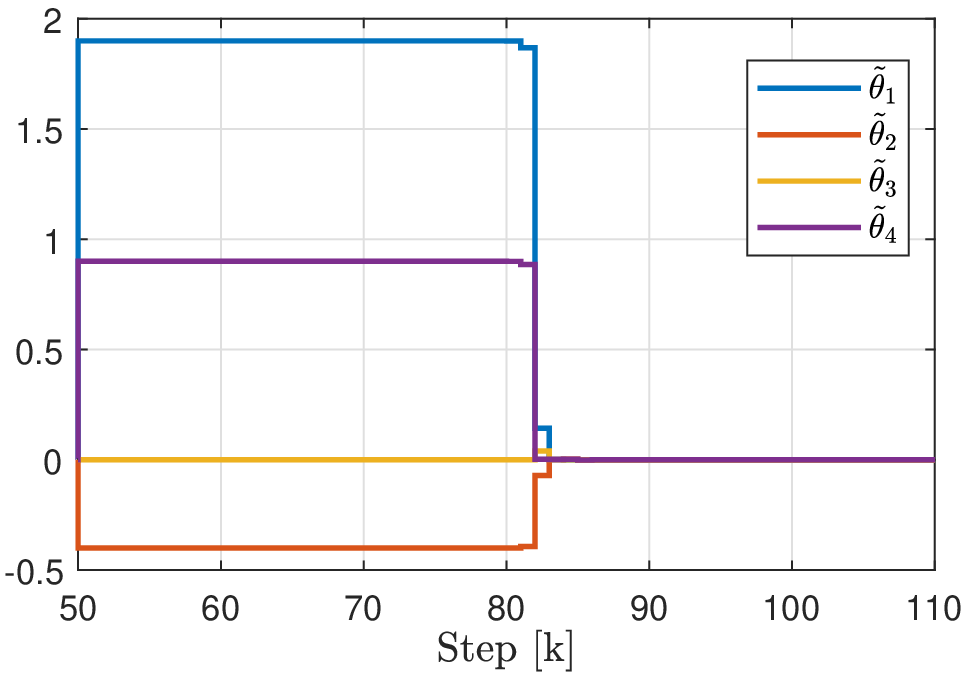}
	}
	\caption{Transients of $\tilde{\theta}$.}
	\label{fig:transients:theta}
\end{figure}

\section{Concluding Remarks}
\lab{sec7}
We have presented in this paper a new robust DREM-based parameter estimator that proves global exponential convergence of the parameter errors with the weakest excitation assumption, namely, identifiability of  the LPRE---which is, actually {\em necessary} for the off- or on-line estimation of the parameters. The main features of the estimator are: (i) it relies on the use of a high performance LS search, in contrast to the usually slower gradient descents; (ii) it ensures component-wise monotonicity of the parameter estimation errors; (iii) it incorporates a forgetting factor avoiding the well-known covariance wind-up problem of LS; (iv) it is applicable to NLPRE, which are separable and monotonic as well as to switching parameters; (v) it constructs the extended regressor avoiding the use of the computationally demanding GPEBO technique, exploiting instead the key structural property of the LS estimator captured in \eqref{prols}; and (vi) CT and DT implementations of the estimator are given. Several simulation results, borrowed from the literature, show the superior performance of the proposed estimator. 

\begcen {\bf  Acknowledgments} \endcen
%
The first author is grateful to Dr. Michelangelo Bin for a detailed explanation of his important results reported in \cite{BIN} and to Prof. Han-Fu Chen for bringing to his attention the use of the mixing step of DREM in \cite{CHEZHAbook}. He also thanks Dr. Lei Wang for his help in the  derivation of the CT version of the LS+D with a forgetting factor.
%
 \begcen{\bf  CRediT authorship contribution statement} \endcen
All authors equally contributed to the paper.


\appendix
\section*{Appendix A}
%
In this appendix we briefly review the main steps in the construction of DREM-based estimators  proceeding from the NLPRE \eqref{nlprect}. For the sake of brevity we restrict ourselves to CT versions, with the DT ones constructed {\em verbatim}. The interested reader is refered to \cite{ORTetaltac} for further details on these constructions.

\subsection*{Derivation of classical DREM-based estimators}

\begenu[{\bf S1}]
\item Starting from the NLPRE $y(t)=\phi^\top(t)\calg(\theta)$, with $y(t) \in \rea,\;\phi(t)\in \rea^p$ measurable signals, $\calg:\rea^q \to \rea^p,\;p\geq q$ and $\theta \in \rea^q$ a constant vector of unknown parameters.
\item (Creation of the extended regressor) Inclusion of a {\em free, stable, linear operator} $\calh:u(t) \to U(t)$, with $u(t) \in \rea$ and $U(t) \in \rea^p$, via its state space realization 
\begali{
\lab{dotu}
\dot U(t) & =A(t)U(t)+b(t)u(t),
}
with $A(t) \in \rea^{p \times p},\;b(t)\in \rea^p$. Upon application to the NLPRE above, create a new extended NLPRE 
\begequ
\lab{extregequ}
Y(t)= \Phi(t)\calg(\theta)
\endequ
with
\begali{
\nonumber
Y(t)&:=\calh[y](t) \in \rea^p\\
\lab{yphiapp}
\Phi(t)&:=[\calh[\phi_1](t) \;|\; \calh[\phi_2](t) \;|\;\dots  \;|\;\calh[\phi_p](t)] \in \rea^{p \times p}.
} 
We underscore the fact that the new extended regressor $\Phi(t)$ is a {\em square matrix}.
\item (Lion's and Kreisselmeier REs) For Lion's RE \cite{LIO} we select for $\calh$ the LTI filter 
$$
A:= \diag\{-a_i\},\; b:=\col(b_1,\dots,b_p),
$$
with $\;b_i \neq b_j,a_i \neq a_j>0,\;(i,j) \in \bar p$.

For Kreisselmeier RE we select LTV operators with
$$
A:= \diag\{-a_i\},\; b(t):=\phi(t),\;a_i >0,\;i \in \bar p.
$$
\item (Mixing step) Multiplication of the extended LPRE \eqref{extregequ} by the {\em adjugate} of $\Phi(t)$ to create the new NLPRE  
\begequ
\lab{scanlpre}
\caly(t) = \Delta(t) \calg(\theta),
\endequ
with
\begalis{
\caly(t)&:=\adj\{\Phi(t)\}Y(t) \in \rea^{p}\\
\Delta(t)&:=\det\{\Phi(t)\} \in \rea
}
and {\em scalar} regressor $\Delta(t)$. Notice that in the case of LPRE $y(t)=\phi^\top(t) \theta$ we obtain $q$ scalar  LPREs of the form
$$
\caly_i(t) = \Delta(t) \theta_i,\;i \in \bar q.
$$
\endenu


\begin{thebibliography}{100}
\bibliographystyle{plain}

\bibitem{ARAetaltac}
S. Aranovskiy, A. Bobtsov, R. Ortega and A. Pyrkin, Performance enhancement of parameter estimators via dynamic regressor extension and mixing,  \TAC, vol. 62, pp. 3546-3550, 2017. 

\bibitem{IIbook}
A. Astolfi, D. Karagiannis and R. Ortega, {\em Nonlinear and Adaptive Control with Applications}, vol. 187, Springer, London, 2008.

\bibitem{BARORT}
N. E. Barabanov and R. Ortega, On global asymptotic stability of $\dot x = \phi(t)\phi^\top (t)x$ with $\phi(t)$ bounded and not persistently exciting, {\it Systems and Control Letters},  vol. 109, pp. 24-27, 2017.

\bibitem{BELetalsysid}
A. Belov, R. Ortega and A. Bobtsov, Guaranteed performance adaptive identification scheme of discrete-time systems using dynamic regressor extension and mixing, {\em 18th IFAC Symposium on System Identification, (SYSID 2018)}, Stockholm, Sweden, July 9-11, 2018.

\bibitem{BIN}
M. Bin, Generalized recursive least squares: Stability, robustness, and excitation, \SCL, vol. 161, 105144, 2022.

\bibitem{BOBetal}
A. Bobtsov, B. Yi,  R. Ortega and A. Astolfi, Generation of new exciting regressors for consistent on-line estimation of a scalar parameter, \TAC, {\tt DOI: 10.1109/TAC.2022.3159568},  2022.

\bibitem{BRUGOEBER}
A. L. Bruce, A. Goel and D. S. Bernstein, Necessary and sufficient regressor conditions for the global asymptotic stability of recursive least squares, \SCL,  vol. 157, 10500527, 2021.

\bibitem{CHEZHAbook}
H. F. Chen and W. Zhao, {\em Recursive Identification and Parameter Estimation}, CRC Press, 2014.

\bibitem{CHOetal}
G. Chowdhary, T. Yucelen, M. Muhlegg and E. N. Johnson, Concurrent learning adaptive control of linear systems with exponentially convergent bounds, {\em International Journal of Adaptive Control and Signal Processing}, vol. 27, no. 4, pp. 280-301, 2013.

\bibitem{CUIGAUANN}
Y. Cui, J. E. Gaudio and A. M. Annaswamy, A new algorithm for discrete-time parameter estimation, {\tt arXiv:2103.16653}, 2021.

\bibitem{DEL}
Ph. de Larminat, On the stabilizability condition in indirect adaptive control, \AUT, vol. 20, no, 6, pp. 793-795, 1984.
 
\bibitem{DEM}
B. P. Demidovich, Dissipativity of nonlinear systems of differential equations, {\it Vestnik Moscow State University, Ser. Mat. Mekh., Part I-6}, (1961) pp. 19-27; {\it Part II-1}, (1962), pp. 3-8, (in Russian).

\bibitem{EFIBARORT}
D. Efimov, N. Barabanov and R. Ortega, Robustness of linear time-varying systems with relaxed excitation, \IJACSP, vol. 33, no. 12, pp. 1885-1900, 2019.

\bibitem{GERORTNIK}
D. Gerasimov, R. Ortega and V. Nikiforov, Adaptive control of multivariable systems with reduced knowledge of high frequency gain: Application of dynamic regressor extension and mixing estimators, {\em 18th IFAC Symposium on System Identification, (SYSID 2018)}, Stockholm, Sweden, July 9-11, 2018.

\bibitem{GOEBRUBER}
A. Goel, A. L. Bruce and D. S. Bernstein, Recursive least squares with variable-direction forgetting: Compensating for the loss of persistency, \CSM, vol. 40, no, 4, pp. 80-102, 2020.

\bibitem{GOOSINbook}
G. Goodwin and K. Sin, {\em Adaptive Filtering Prediction and Control}, Prentice-Hall, 1984.

\bibitem{KORetalecc20}
M. Korotina,  S. Aranovskiy, R. Ushirobira and A. Vedyakov, On parameter tuning and convergence properties of the DREM procedure, {\em 2020 European Control Conference (ECC20)}, Saint Petersburg, Russia, May 12-15, 2020.

\bibitem{KORetal}
M. Korotina, J. G. Romero, S. Aranovskiy, A. Bobtsov and R. Ortega, Persistent excitation is unnecessary for on-line exponential parameter estimation: a new algorithm that overcomes this obstacle, \SCL, vol. 159, {\tt doi.org/10.1016/j.sysconle.2021.105079}, 2022. 

\bibitem{KRAKHA}
J. Krause and P. Khargonekar, Parameter information content of measurable signals in direct adaptive control, {\em IEEE Trans. on Automatic Control}, vol. 32, no. 9, pp. 802-810. 1987.

\bibitem{KRERIE}
G. Kreisselmeier and G. Rietze-Augst, Richness and excitation on an interval---with application to continuous-time adaptive control, \TAC, vol. 35, no. 2, pp. 165-171, 1990.

\bibitem{KRE}
G. Kreisselmeier, Adaptive observers with exponential rate of convergence, \TAC, vol. 22, no. 1, pp. 2-8, 1977.

\bibitem{KRls}
M. Krstic,  On using least-squares updates without regressor filtering in identification and adaptive control of nonlinear systems, {\em Automatica}, vol. 45, pp. 731-735, 2009. 

\bibitem{LIBbook}
D. Liberzon, {\em Switching in Systems and Control}, vol 190, Springer, 2003.

\bibitem{LIO}
P.M. Lion, Rapid identification of linear and nonlinear systems, {\em AIAA Journal}, vol. 5, pp. 1835-1842, 1967.

\bibitem{LJUbook} 
L. Ljung, {\em System Identification: Theory for the User}, Prentice Hall, New Jersey, 1987.

\bibitem{LJUSODbook}
L. Ljung and T. Soderstrom, {\em Theory and Practice of Recursive Identification}, MIT Press, 1983.

\bibitem{LOZZHA}
R. Lozano and X. Zhao, Adaptive pole placement without excitation probing signals, \TAC, vol. 39, no. 1, pp. 47-58, 1994.

\bibitem{LOZCAN}
R. Lozano and C. Canudas, Passivity-based adaptive control of mechanical manipulators using LS-type estimation, \TAC, vol. 35, no. 12, pp. 1363-1365, 1990.

\bibitem{MARTOM}
R. Marino and P. Tomei, On exponentially convergent parameter estimation with lack of persistency of excitation, \SCL,  vol. 159, 105080, 2022.

\bibitem{NIBK} 
T. Nguyen, S. Islam, D. Bernstein and I. Kolmanovsky, Predictive cost adaptive control: A numerical investigation of persistency, consistency, and exigency, {\em IEEE Control Systems Magazine}, vol. 41, pp. 64-96, 2021.

\bibitem{ORTproieee}
R. Ortega,  An on-line least-squares parameter estimator with finite convergence time, {\it Proc. IEEE}, vol. 76, no. 7, 1988.

\bibitem{ORTetal_pebo}
 R. Ortega, A. Bobtsov, A. Pyrkin and S. Aranovskiy, A parameter estimation approach to  state observation of nonlinear systems, {{\it Systems and Control Letters}},  vol. 85, pp 84-94, 2015.
 
\bibitem{ORTetal_aut19}
R. Ortega, D. Gerasimov, N. Barabanov and V. Nikiforov, Adaptive control of linear multivariable systems using dynamic regressor extension and mixing estimators: Removing the high-frequency gain assumption, \AUT, vol. 110, 108589, 2019.

\bibitem{ORTNIKGER}
R. Ortega, V. Nikiforov and D. Gerasimov, On modified parameter estimators for identification and adaptive control: a unified framework and some new schemes, \ARC, vol. 50, pp. 278-293, 2020.

\bibitem{ORTetal_gpebo}
R. Ortega, A. Bobtsov, N. Nikolayev, J. Schiffer and D. Dochain, Generalized parameter estimation-based observers: Application to power systems and chemical-biological reactors,  \AUT,  vol. 129, 109635, 2021.

\bibitem{ORTetalaut21}
R. Ortega, V. Gromov, E. Nu\~no, A. Pyrkin and J. G. Romero, Parameter estimation of nonlinearly parameterized regressions: application to system identification and adaptive control,  \AUT,  vol. 127, 109544, 2021.	

\bibitem{ORTajc}
R. Ortega, Comments on recent claims about trajectories of control systems valid for particular initial conditions,  \AJC, {\tt DOI: 10.1002/asjc.2512}, 2021.

\bibitem{ORTetaltac}
R. Ortega, S. Aranovskiy, A. Pyrkin, A Astolfi and A. Bobtsov, New results on parameter estimation via dynamic regressor extension and mixing: Continuous and discrete-time cases, \TAC, vol. 66, no. 5, pp. 2265-2272, 2021.

\bibitem{ORTBOBNIK}
R. Ortega, A. Bobtsov and  N. Nikolayev, Parameter identification with finite-convergence time alertness preservation, \CSL, vol. 6, pp. 205-210, 2022

\bibitem{PANYU}
Y. Pan and H. Yu, Composite learning robot control with guaranteed parameter convergence, {\em Automatica}, vol. 89, pp. 398-406, 2018.

\bibitem{PANetal}
Y. Pan, S. Aranovskiy, A. Bobtsov, and H. Yu, Efficient learning from adaptive control under sufficient excitation, {\em International Journal of Robust and Nonlinear Control}, vol. 29, pp. 3111-3124, 2019.

\bibitem{PAVetal}
A. Pavlov, A. Pogromsky, N. van de Wouw and H. Nijmeijer, Convergence dynamics, a tribute to Boris Pavlovich Demidovich, \SCL, vol. 52, pp. 257-261, 2004.

\bibitem{PRA17}
L. Praly, Convergence of the gradient algorithm for linear regression models in the continuous and discrete-time cases, {\em Int. Rep. MINES ParisTech}, Centre Automatique et Syst\`emes, December 26, 2017.

\bibitem{PYRetal}  
A. Pyrkin, R. Ortega, V. Gromov, A. Bobtsov and A. Vedyakov, A Globally convergent direct adaptive pole-placement controller for nonminimum phase systems with relaxed excitation assumptions, \IJACSP, vol. 33, pp 1491-1505, 2019.
	
\bibitem{RAOTOUbook}  
C. Rao and H. Toutenburg, H, {\em Linear Models: Least Squares and Alternatives}, Springer Series in Statistics (3rd ed.). Berlin, Springer, 2008.
	
\bibitem{RUD}
W. Rudin, {\em Principes of Mathematical Analysis}, 3rd ed., NY:McGraw-Hill, Inc, 1976.

\bibitem{RUGbook}
W.J. Rugh, {\em Linear Systems Theory}, 2nd ed., Prentice hall, NJ, 1996.

\bibitem{SASBODbook}
S. Sastry and M. Bodson, {\em Adaptive Control: Stability, Convergence and Robustness}, Prentice-Hall, New Jersey, 1989.

\bibitem{SHAetal}
V. Shaferman, M. Schwegel, T. Gluck and A. Kugi, Continuous-time least-squares forgetting algorithms for indirect adaptive control, \EJC, vol. 62, pp.105-112, 2021.

\bibitem{SHILEE}
H. Shin and H. Lee, A new exponential forgetting algorithm for recursive least-squares parameter estimation, {\tt arXiv:2004.03910}, 2020. 

\bibitem{SLOLIbook}
J.-J. E. Slotine  and W. Li, {\em Applied Nonlinear Control}, Prentice-Hall, New Jersey, USA, 1991.
	
\bibitem{TAObook}
G. Tao, {\em Adaptive Control Design and Analysis}, vol. 37,  John Wiley \& Sons, New Jersey, 2003.

\bibitem{WANetal}
L. Wang, R. Ortega, A. Bobtsov, J. G. Romero and B. Yi, Identifiability implies robust, globally exponentially convergent on-line parameter estimation: Application to model reference adaptive control, {\tt arXiv:2108.08436}, 2021.

\bibitem{WUetal}
Z. Wu, M. Ma, X. Xu, B. Liu and Z. Yu. Predefined-time parameter estimation via modified dynamic regressor extension and mixing, {\em Journal of the Franklin Institute}, {\tt doi.org/10.1016/j.jfranklin.2021.06.028}, 2021.

\bibitem{YIORT} 
B. Yi and R. Ortega, Conditions for convergence of dynamic regressor extension and mixing parameter estimators using LTI filters, \TAC, {\tt 10.1109/TAC.2022.3149964}, 2022.  

\end{thebibliography}
\end{document}